\documentclass[12pt]{article}

\usepackage{amsmath, amsthm, amssymb, graphicx, nicefrac, footnote}
\usepackage{url}
\usepackage[small,nohug,heads=vee]{diagrams}
\diagramstyle[labelstyle=\scriptstyle]

\usepackage[table]{xcolor}
\usepackage{xy}
\usepackage{tikz}

\definecolor{cof}{RGB}{219,144,71}
\definecolor{pur}{RGB}{186,146,162}
\definecolor{greeo}{RGB}{91,173,69}
\definecolor{greet}{RGB}{52,111,72}

\newtheorem{teo}{Theorem}[section]
\newtheorem{lem}[teo]{Lemma}
\newtheorem{pro}[teo]{Proposition}
\newtheorem{cor}[teo]{Corollary}
\newtheorem{Rem}[teo]{Remark}
\newtheorem{ex}[teo]{Example}
\newtheorem{defi}[teo]{Definition}

\newcommand{\nat}{\mathbb{N}}

\newcommand{\inte}{\mathbb{Z}}

\newcommand{\field}{\mathbb{K}}
\newcommand{\Ffield}{\mathbb{K}}
\newcommand{\di}{\operatorname{dim}}

\title{Syzygies of the Veronese modules}
\author{Ornella Greco, Ivan Martino}

\date{\empty}

\begin{document}
\maketitle

\begin{abstract}
\begin{small}
\noindent
We study the minimal free resolution of the Veronese modules, $S_{n,d,k}=\oplus_{i\ge0}S_{k+id}$, where $S=\field[x_1,\ldots,x_n]$, by giving a formula for the Betti numbers in terms of the reduced homology of some skeleton of a simplicial complex. 
We prove that $S_{n,d,k}$ is Cohen-Macaulay if and only if $k<d$, and that its minimal resolution is pure and has some linearity features when $k>d(n-1)-n$.
We prove combinatorially that the resolution of $S_{2,d,k}$ is pure.
We show that $HS(S_{n,d,k}; z)= \frac{1}{(n-1)!} \frac{d^{n-1}}{dz^{n-1}}\left[ \frac{z^{k+n-1}}{1-z^d}\right]$.
As an application, we  calculate the complete Betti diagrams of the Veronese rings $\field[x,y,z]^{(d)}$, for $d=4,5$, and $\field[x,y,z, u]^{(3)}$. 
%

\end{small}
\end{abstract}

\vspace{5 mm}
\noindent
Given a graded ring $S=\oplus_{i\geq 0}S_i$, the Veronese subring $S^{(d)}$ is defined as $\oplus_{i\geq 0}S_{id}$ and the \emph{Veronese modules} $S_{n,d,k}$, which are modules over the Veronese subring, are $\oplus_{i\geq 0}S_{k+id}$.

\noindent
In this paper, we set $S=\field[x_1, \dots , x_n]$, where $\field$ is a field,
 and we deal with the syzygies of the \emph{Veronese modules}.
Since $S^{(d)}$ can be presented as $R/I$, where $R$ is a polynomial ring and $I$ a binomial ideal, in the following, we will consider $S_{n,d,k}$ as an $R$-module (see Section \ref{sec-Preliminaries}).

There has been a lot of effort already to find  the graded Betti numbers of the Veronese ring $S^{(d)}$. The problem can be really hard, namely, in \cite{EinLazarsfeld}, Ein and Lazarsfeld showed that for $d\gg 0$ the graded Betti numbers $\beta_{i,j}(S^{(d)})\neq 0$ for many $j$ if $i$ is large: in particular, they proved that, for $ d\gg 0$ there exist $l_1,l_2$ such that $\beta_{p,p+q}\neq 0$ for all $p$ in the range $
l_1 d^{q-1}\leq p\leq {d+n\choose n}-l_2 d^{n-q}$.

It is known (see references in \cite{OttavianiPaolettiVeronese})  that  $\beta_{i}=\beta_{i,(i+1)d}$, for all $i>0$, in the cases $n=2$ or $(d,n)=(2,3)$. Instead, when $d=2$ and $n>3$, we have that the equality holds for $i\leq 5$. 
In addition, for $d=2$, all Betti numbers have been determined.
In case $n,d\geq 3$, Ottaviani and Paoletti (in \cite{OttavianiPaolettiVeronese}) also proved that $\beta_{i}\neq \beta_{i,(i+1)d}$, for all $i>3d-3$, and conjectured that the equality holds for $i\leq 3d-3$. 
They proved their conjecture for $n=3$ and $(d,n)=(3,4)$.
Bruns, Conca and R\"omer, in \cite{BrunsConcaRomer}, provided another proof for $n=3$.

Moreover, in \cite{GotoWatanabe}, Goto and Watanabe proved that the Veronese module $S_{n,d,k}$ is Cohen-Macaulay  when $k<d$,  and that the canonical module of $S^{(d)}$ is given by the Veronese module $S_{n,d,d-n}$ (these results can also be  found in   \cite[Section 3]{BrunsHerzog}).

Furthermore, in \cite{AramovaBarcanescuHerzog}, Aramova, B{\u{a}}rc{\u{a}}nescu and Herzog
proved that the resolution of $S_{n,d,k}$ as $S^{(d)}$-module is linear.

In \cite{CampilloMarijuan}, Campillo and Mariju\'{a}n showed a way to compute the Betti numbers of numerical semigroup rings in terms of homology of a certain simplicial complex.
Later, in order to calculate the Betti numbers of affine semigroup rings, Bruns and Herzog, in \cite{BrunsHerzogSemi}, reintroduced Campillo and Mariju\'{a}n's simplicial complex, calling it \emph{squarefree divisor complex}, $\Delta_\mathbf{c}$ (see Section \ref{sec-Preliminaries}). 
In particular, they gave a formula for $\beta_{i,j}(S^{(d)})$ in terms of the dimension of the reduced homology of $\Delta_\mathbf{c}$. 
Recently, in \cite{Spaul}, Paul gave another description of the graded Betti numbers of semigroup rings (actually he worked in a more general environment) in terms of the reduced homology of a simplicial complex, $\Gamma_\mathbf{c}$, called \emph{pile simplicial complex}.

In this paper, we concentrate on the syzygies of $S_{n,d,k}$. We will use combinatorial methods, and precisely we will relate the Betti numbers of the Veronese modules to the simplicial complex introduced by Paul. We would like to point out that some of these results may be as well obtained using alternative approaches and using methods, like the Koszul cohomology. We are pursuing  the combinatorial approach, due to the simple proofs that it leads us to. It is worth to mention a few papers that deal with syzygies of commutative rings using similar simplicial complexes: for instance, Vu, in \cite{ThanhVu}, uses squarefree divisor complexes for proving the Koszul property of the pinched Veronese varieties; similiar approaches can also be found in \cite{CampilloGimenez} and \cite{Sturmfels}, in relation with toric ideals, and in \cite{KaraReinerWachs} and \cite{Stanley} in relation, respectively, with chessboard complexes and with $\mathbb{N}$-solutions to linear equations.\\

The key result of the paper generalizes Paul's and Bruns-Herzog's formulas to the Veronese modules.

\setcounter{section}{3}\setcounter{teo}{0}
\begin{teo}
Let $S=\oplus_{i\geq 0} S_i$ and  $S_{n,d,k}=\oplus_{i\geq 0}S_{k+id}$.
If $\mathbf{c}$ is a vector in $\inte^n$ such that $|\mathbf{c}|=k+jd$, then 
\[
  \beta_{i,\mathbf{c}}(S_{n,d,k})=\operatorname{dim}_{\Ffield} \tilde{\operatorname{H}}_{i-1}(\Gamma_{\mathbf{c}}^{\langle j-1 \rangle}, \Ffield),
\]
where $\Gamma_{\mathbf{c}}^{\langle j-1 \rangle}$ is the $(j-1)$-skeleton of $\Gamma_{\mathbf{c}}$. Moreover, $\beta_{i,\mathbf{c}}(S_{n,d,k})=0$ when  $|\mathbf{c}|\neq k+jd$, for all $j$.
\end{teo}

\noindent
Using this tool, in Theorem \ref{theorem-regularity},  we show that the Betti numbers of $S_{n,d,k}$ can be non-zero only in degrees $k+id$ for $i< n$; we also characterize when these modules  are Cohen-Macaulay.

\setcounter{section}{3}\setcounter{teo}{4}
\begin{teo}
The Veronese module $S_{n,d,k}$ is Cohen-Macaulay if and only if $k<d$.
Moreover if $S_{n,d,k}$ is not Cohen-Macaulay, then it has maximal projective dimension, that is $\binom{d+n-1}{d}-1$.
\end{teo}


Later, in Theorem \ref{theorem-linear-resolution}, we prove that if $k>d(n-1)-n$, then the resolution of the Veronese module $S_{n,d,k}$ is pure (and actually $\beta_i=\beta_{i,k+id}$).

We also find a general way to compute the rational form of the Hilbert series of the Veronese modules. 
Indeed we prove that:
\setcounter{section}{2}\setcounter{teo}{0}
\begin{teo}
$\frac{d}{dz} HS(S_{n,d,k}; z) =n HS(S_{n+1,d,k-1}; z)$.
\end{teo}
\noindent
Hence,
\[
  HS(S_{n,d,k}; z)= \frac{1}{(n-1)!} \frac{d^{n-1}}{dz^{n-1}}\left[ \frac{z^{k+n-1}}{1-z^d}\right].
\]
This allows us to write a closed formula for $HS(S_{n,d,k}; z)$  for $n\leq 3$ (see equations (\ref{eq:HS-2}) and (\ref{eq:HS-3})) and, by 
differentiating, one could get the Hilbert series for larger $n$.

%
Moreover, in Section \ref{sec-n=2} we give an alternative proof to the following theorem about the Betti table for $S_{2,d,k}$, for $k < d$.
\setcounter{section}{4}\setcounter{teo}{0}
\begin{teo}\label{theo-jump}
  If $k<d$, the Veronese module $S_{2,d,k}$ has pure resolution and the Betti table is:
  \[
  \begin{array}{c|cccccccc}
    &0 &1 &\dots &k &k+1 &k+2    &\dots& d-1\\\hline
    k &k+1 &k\binom{d}{1}& \dots &\binom{d}{k}&0&0&\dots&0\\
 
    k+1 &0 &0& \dots&0&\binom{d}{k+2}&2\binom{d}{k+3}&\dots&(d-1-k)\binom{d}{d}
  \end{array}
  \]
  Namely,  $\beta_{i}(S_{2,d,k})=\beta_{i, k+id}(S_{2,d,k})=(k+1-i)\binom{d}{i}$ for $i\leq k$, and $\beta_{i}(S_{2,d,k})=\beta_{i, k+(i+1)d}(S_{2,d,k})=(i-k)\binom{d}{i+1}$ for $i>k$.
\end{teo}
\noindent
This result can be obtained as a consequence of Corollary (3.a.6) in \cite{Green_koszul}.\\

Finally, for $k>0$, we prove the linearity of the first step of the minimal resolution of $S_{n,d,k}$ (see Corollary \ref{cor-linearity-general-n-d-k}).\\

The first section provides a summary of results about the Veronese rings and the definition of Veronese modules.
In the second section we concentrate on their Hilbert series. 
Later, in Section 3, we prove our theoretical result on the Betti numbers of the Veronese modules, we characterize the cases in which they are Cohen-Macaulay, and we give a sufficient condition for the linearity of their minimal graded free resolution.
In Section 4, we deal with the case $n=2$, describing  the Betti diagrams of $S_{2,d,k}$.
Finally, in Section 5, we calculate the Betti tables of $S_{3,4,0}$, $S_{3,5,0}$ and $S_{4,3,0}$.

\setcounter{section}{0}\setcounter{teo}{0}
\section{Preliminaries}\label{sec-Preliminaries}
In this section, we recall the definition of Veronese subring and Veronese modules. We give  also a short summary of  some of the results, known in literature, that relate these with the \emph{squarefree divisor complex}, given by Campillo and Mariju\'{a}n in \cite{CampilloMarijuan} and by Bruns and Herzog in \cite{BrunsHerzogSemi}, and with the \emph{pile simplicial complex}, given by  Paul in \cite{Spaul}.

\begin{defi}
 Let $\mathcal{A}$ be the set
 $\{(a_1,a_2,\dots ,a_n)\in \mathbb{N}^n| \ \sum_{i=1}^n a_i=d\}$.
 The \emph{Veronese subring} of $S$ is the algebra $S^{(d)}=\field[\mathbf{x}^\mathbf{a}| \ \mathbf{a} \in \mathcal{A}]$.  
\end{defi}

\noindent
A presentation of $S^{(d)}$ is given by 
\begin{eqnarray*}
  \phi:\field[y_1,\dots , y_N] &\rightarrow& \field[x_1,\dots , x_n]\\
			    y_i&\mapsto & \mathbf{x}^{\mathbf{a}_i}
\end{eqnarray*}
with $\mathbf{a}_i$ being the $i$-th  element  of $\mathcal{A}$ with respect to the lexicographic order, and $N=\binom{d+n-1}{d}$ being the cardinality of $\mathcal{A}$. From now on, we will use the notation $R$ for $\field[y_1,\dots , y_N]$, and $S$ for $\field[x_1,\dots , x_n]$.
Thus, $S^{(d)}\cong\frac{R}{\textrm{ker}\phi}$.\\

Let us consider the affine numerical semigroup $H\subseteq \mathbb{N}^{n}$ generated by the set $\mathcal{A}$. 
\begin{defi}
  Given an element $\mathbf{h}\in H$, we define the \emph{squarefree divisor complex} to be the simplicial complex  
  \[
    \Delta_{\mathbf{h}}=\left\{\{\mathbf{a}_{i_1},\dots \mathbf{a}_{i_k}\} \subseteq \mathcal{A} |\ \mathbf{x}^{\mathbf{a}_{i_1}+\dots+\mathbf{a}_{i_k}} \mbox{ divides } \mathbf{x}^{\mathbf{h}}\right\}.
  \]
%
%
\end{defi}

\noindent
The following result was proved by Bruns and Herzog (see \cite[Proposition 1.1]{BrunsHerzogSemi}) in a more general setting, here we are only stating the version for the Veronese subrings.
\begin{teo}
 Let $i\in \mathbb{Z}$ and $h\in H$, then
  \[
    \beta_{i,\mathbf{h}}(S^{(d)})= \operatorname{dim}_\Ffield \tilde{\operatorname{H}}_{i-1}(\Delta_{\mathbf{h}},\Ffield).
  \]
\end{teo}

Let us define the partial ordering in $\mathbb{Z}^{n}$ as $\mathbf{a}\leq \mathbf{b}$ if and only if $\mathbf{b}-\mathbf{a}\in \mathbb{N}^{n}$.

\begin{defi}
 The \emph{pile simplicial complex} of $\mathcal{A}$ is
 $$
 \Gamma_{\mathbf{c}}=\{ F\subseteq \mathcal{A}| \ \sum_{\mathbf{a}\in F}\mathbf{a}\leq \mathbf{c} \}.
 $$
\end{defi}

\noindent
This simplicial complex is equal to the squarefree divisor complex, when $\mathbf{c}$ belongs to the semigroup generated by $\mathcal{A}$.\\
 
\noindent
In \cite[Theorem 1]{Spaul}, Paul first proved a duality formula (proved previously in \cite{dong}), namely:
\begin{equation}\label{duality}
   \tilde{\operatorname{H}}_{i-1}(\Gamma_{\mathbf{c}},\Ffield)\cong \tilde{\operatorname{H}}_{N-n-i-1}(\Gamma_{\hat{\bf{c}}},\Ffield)^\vee,
  \end{equation}
  where $\bf{t}=\sum_{\bf{a}\in \mathcal{A}}\bf{a}$ and $\hat{\bf{c}}=\bf{t}-\bf{c}-\mathbf{1}$.
Then, in \cite[Theorem 7]{Spaul}, he  applied the isomorphism above to obtain the following result.
\begin{teo}
  Let $i\in \mathbb{Z}$ and $\mathbf{c}\in \mathbb{Z}^{n}$, then
  \[
    \beta_{i,\bf{c}}(S^{(d)})= \operatorname{dim} \tilde{\operatorname{H}}_{N-n-i-1}(\Gamma_{\hat{\bf{c}}},\Ffield).
  \]
  \end{teo}

From now on, given a vector $\mathbf{c}=(c_1,\dots, c_n)$  in $\inte^n$,  $|\mathbf{c}|
=c_1+c_2+\dots+c_n$  denotes the total degree  of $\mathbf{c}$. 


%


\begin{defi}
  Let $n,d,k \in \mathbb{N}$, the \emph{Veronese modules}, $S_{n,d,k}$,  are defined as $S_{n,d,k}=\oplus_{i\geq 0} S_{k+id}$.
\end{defi}
%


By the Auslander-Buchsbaum formula (see \cite{BrunsHerzog, Twenty-four}), we have that
$$
\mathrm{pdim}(S_{n,d,k})+\mathrm{depth}(S_{n,d,k})=\operatorname{depth}(R)
$$
and $\operatorname{depth}(R)=\operatorname{dim}(R)=\operatorname{emb}(S^{(d)})$.
Moreover $y_1$ is always a non zero-divisor with respect to $S_{n,d,k}$ and so 
\[
  1\leq \operatorname{depth} (S_{n,d,k})\leq \operatorname{dim} (S_{n,d,k})=n
\]
and then
\[
  N-n\leq \operatorname{pdim} (S_{n,d,k})\leq N-1.
\]
In particular $S_{n,d,k}$ is Cohen-Macaulay if and only if  it has depth $n$, i.e. projective dimension  $N-n$.\\

In the following sections, we are going to state that some resolutions are linear even if they are not according to the standard definition (see \cite{Villarreal}).
\begin{defi}
Let $A$ be a polynomial ring, $I$ a graded ideal in $A$, and let $T=A/I$. Consider the minimal free resolution of $T$ by free $A$-modules:
$$
0 \rightarrow \oplus_{i=1}^{\beta_p}A(-d_{pi})\rightarrow \cdots \rightarrow \oplus_{i=1}^{\beta_1}A(-d_{1i})\rightarrow A \rightarrow T \rightarrow 0,
$$
the ideal $I$ has a \emph{pure} resolution if there are $d_1, \dots , d_p$, with $d_i<d_{i+1}$,
such that $d_{1i}=d_1, \dots , d_{pi}=d_p $ for all $i$. 
\end{defi}

We recall the definition of linearity of the resolution of a module, as given by Eisenbud and Goto in \cite{EisenbudGoto}.
\begin{defi}
Let $M$ be a finitely generated graded $S$-module. The module $M$ has
\emph{$p$-linear resolution}, over the polynomial ring $S$, if its minimal free graded resolution has the form:
$$
\cdots \rightarrow S(-p-i)^{\beta_i}\rightarrow \cdots \rightarrow S(-p-1)^{\beta_1} \rightarrow S(-p)^{\beta_0} \rightarrow M \rightarrow 0.
$$
\end{defi}

\begin{defi}
The module $S_{n,d,k}$ has a \emph{pseudo-linear} resolution if its mini\-mal free resolution is pure and, in addition, $\beta_i=\beta_{i, k+id}$ for all $i$.
\end{defi}
\begin{Rem}
Throughout the paper, we are letting the generators of $S_{n,d,k}$ have degree $k$ and the generators of $S^{(d)}$ have degree $d$. If instead we choose $0 $ for the degree of generators of the module $S_{n,d,k}$  and $1$ for the degree of the generators of $S^{(d)}$, we would have that the resolution is pseudo-linear when $\beta_i(S_{n,d,k})=\beta_{i,i}(S_{n,d,k})$. In the last case, the definition of pseudo-linearity coincides with the definition of linearity, given in \cite{EisenbudGoto}, in fact $S_{n,d,k}$ would have a $0$-linear resolution.
\end{Rem}
Finally, let us recall a result by Goto and Watanabe on the canonical module of Veronese ring.
\begin{teo}[Corollary 3.1.3 in \cite{GotoWatanabe}]
The canonical module of $S^{(d)}$ is $S_{n,d,d-n}$.
\end{teo}

\section{The Hilbert series of the Veronese modules}\label{sec-HS}

Let us fix some notation. 
We denote by $HS(M; z)$ the Hilbert series of the module $M$.
The Hilbert series of $S_{n,d,k}$ as an $R$-module
is equal to $h(z)/(1-z)^N$, where in the numerator we have the polynomial $P(z)=\sum_{i,j}(-1)^i \beta_{i,j}z^j$.

In the literature, there has been some work trying to find an explicit formula for the Hilbert series and Hilbert polynomial of the Veronese rings. 
Recently, Brenti and Welker showed (see  \cite[Theorem 1.1]{BrentiWelker}) that the Hilbert series of $S_{n,d,0}$ is
\[
  HS(S_{n,d,0}; z)=\frac{\sum_{i=0}^{n} C(d-1, n, id) z^{id}}{(1-z^d)^n},
\]
where 
$C(d-1, n, id)=\#\{\mathbf{a}\in \mathbb{N}^n: |\mathbf{a}|=id, a_j\leq d-1 \mbox{ for any }j\}$.
One could compute the polynomial  $P(z)$ by multiplying the numerator by $(1-z^d)^{N-n}$.
%

%
One observes that $h(z)=\sum_{i=0}^{\alpha} N_i z^{k+id}$, where $\alpha=\lfloor \frac{N(d-1) -k}{d}\rfloor$ and, using Remark 3.2 in \cite{One} one knows that if $k<d$
\[
    N_i=\sum_{s=0}^{i} (-1)^{s} \binom{N}{s}\binom{N-1+d(i-s)+k}{N-1}.
\]

In one variable it is easy to see that the Hilbert Series of $S_{1,d,k}$ is
\begin{equation}\label{eq:hilbert-one-variable}
  HS(S_{1,d,k}; z)=\frac{z^k}{1-z^d}.
\end{equation}
We want to find a direct formula for the Hilbert series of $S_{n,d,k}$ and we use the following property.

\begin{teo}\label{theo:differentiation-HS}
$\frac{d}{dz} HS(S_{n,d,k}; z) =n HS(S_{n+1,d,k-1}; z)$.
\end{teo}
\begin{proof}
By definition of $S_{n,d,k}$ the Hilbert series is
\[
  HS(S_{n,d,k}; z)=\sum_{i\geq 0}\binom{k+id+(n-1)}{n-1}z^{k+id}.
\]
Let us consider the first derivative
\[
  \frac{d}{dz}HS(S_{n,d,k}; z)=\sum_{i\geq 0}\binom{k+id+(n-1)}{n-1} (k+di)z^{k+id-1}
\]
and we analyze the coefficient of $z^{k+id-1}$, i.e.
\begin{eqnarray*}
  \binom{k+id+(n-1)}{n-1} (k+di)&=&\frac{(k+id+n-1)!}{(n-1)!(k+id)!}(k+id)\\
  &=&n\frac{(k+id+n-1)!}{(n)!(k+id-1)!}\\
  &=&n\binom{(k-1)+id+n}{n}.
\end{eqnarray*}
\end{proof}

\noindent
As a consequence, we get a direct expression:
\begin{cor}\label{cor-close-formula-HS}
The Hilbert series of the Veronese modules is
%
\[
  HS(S_{n,d,k}; z)= \frac{1}{(n-1)!} \frac{d^{n-1}}{dz^{n-1}}\left[ \frac{z^{k+n-1}}{1-z^d}\right].
\]
\end{cor}
\noindent
Therefore, by differentiating $\frac{z^{k+n-1}}{1-z^d}$ one finds $HS(S_{n,d,k}; z)$.\\

\noindent
Using the computer program \texttt{Maple}, 
\begin{verbatim}
  (1/(n-1)!)*diff(z^(k+n-1)/(1-z^d),z$(n-1));
\end{verbatim}
one could compute the Hilbert series up to $n=95$ in $0.970$ second. (We are using a \texttt{Dell OptiPlex 790 with Intel Core i7-2600 (3.40GHz, 8MB) and 16 GB memory, Ubuntu 12.04.4 LTS 64-bit}). In the paper, though, we are going to deal only with the cases $n=2,3$, since, already for $n=4$, the formulas of the Hilbert series and of the polynomial $P(z)$ are quite long and not elegant.\\

Let us write down the general formula for $HS(S_{2,d,k}; z)$ and $HS(S_{3,d,k}; z)$:
\begin{equation}\label{eq:HS-2}
  HS(S_{2,d,k}; z)=\frac{z^k[1+k+(d-k-1)z^d]}{(1-z^d)^2}
\end{equation}
and 
\begin{equation}\label{eq:HS-3}
\begin{split}
HS(S_{3,d,k}; z) &=\frac{{z}^{k}(k+1)(k+2)}{2(1-z^d)^3}\\
  &+\frac{{z}^{k+d}[-2(k+1)(k+2)+d(2k+3+d)]}{2(1-z^d)^3}\\
  &+\frac{{z}^{k+2d}[(k+1)(k+2)-d(2k+3-d)]}{2(1-z^d)^3}
\end{split}
\end{equation}




\noindent
Therefore one could compute the polynomials $P(z)$.
\noindent
Namely, the polynomial $P(z)$ of $S_{2,d,k}$ is
\[
  P(z)=\sum_{i=0}^{d} (-1)^{i+1} (i-(k+1))\binom{d}{i} z^{k+id}.
\]


\noindent
Moreover, the polynomial $P(z)$ of $S_{3,d,k}$ is
\[
  \sum_{i=0}^{N-1}  \frac{(-1)^{i-2}}{2} \left(a\binom{N-3}{i}-(b-2a)\binom{N-3}{i-1}+(a-c)\binom{N-3}{i-2}\right) z^{k+id}.
\]
where $a=(k+1)(k+2)$, $b=d(2k+3+d)$, $c=d(2k+3-d)$ and $N={{d+2}\choose 2}$.

We are going to use these polynomials to compute the Betti numbers of the Veronese module in the next sections.



\section{The Betti table of the Veronese modules}\label{section:general results}

This section contains our main theorem, which gives the connection between  the syzygies of the Veronese modules and the pile simplicial complex. 

\begin{teo}\label{Theorem-theoretical-results}
If $\mathbf{c}$ is a vector in $\inte^n$ such that $|\mathbf{c}|=k+jd$, 
then $$\beta_{i,\mathbf{c}}(S_{n,d,k})=\operatorname{dim}_{\Ffield} \tilde{\operatorname{H}}_{i-1}(\Gamma_{\mathbf{c}}^{\langle j-1 \rangle}, \Ffield),$$
where $\Gamma_{\mathbf{c}}^{\langle j-1 \rangle}$ is the $(j-1)$-skeleton of $\Gamma_{\mathbf{c}}$. Moreover, $\beta_{i,\mathbf{c}}(S_{n,d,k})=0$ when  $|\mathbf{c}|\neq k+jd$, for all $j$. 
\end{teo}
\begin{proof}
In order to compute the Betti numbers of $M=S_{n,d,k}$, we need to consider the homology
of  $\mathbb{K}$, the Koszul complex
of $M$. 
The $i$-th module in the Koszul complex is denoted with $\mathbb{K}_i$ and it is equal to $ \wedge^i M=\oplus M e_{j_1}\wedge \cdots \wedge e_{j_i}$, so its non zero graded components lie in degrees $(k+id,k+(i+1)d, \cdots)$.\\
Given a multidegree $\mathbf{c}$ such that $|\mathbf{c}|\neq k+jd$, for all $j$, one notices that $\di(\mathbb{K}_i)_\mathbf{c}=0$, i.e. $\beta_{i,\mathbf{c}}(M)=0$, for all $i$.\\
Now, let us take $\mathbf{c}$ with $|\mathbf{c}|= k+jd$, for some $k \in \nat$, we aim to prove that:
$$
(\mathbb{K}_i)_\mathbf{c}\cong \tilde{C}_{i-1}(\Gamma_{\mathbf{c}}^{\langle j-1 \rangle}, \Ffield),
$$
where $\tilde{C}_{i-1}(\Gamma_{\mathbf{c}}^{\langle j-1 \rangle})$ is the $(i-1)$-th chain group.
We notice that $\di (\mathbb{K}_i)_\mathbf{c} \neq 0$ if and only if $i\leq j$; similarly, by the definition of skeleton,  $\di (\tilde{C}_{i-1}(\Gamma_{\mathbf{c}}^{\langle j-1 \rangle}))\neq 0$
 if and only if $i\leq j$. This implies that $\beta_{i,k+jd}(M)=0$, for all $i>j$.\\
 Let us consider the case $i\leq j$, and let $0 \neq m e_{j_1}\wedge \cdots \wedge e_{j_i} \in (\mathbb{K}_i)_\mathbf{c}$, where $m=\mathbf{x}^\mathbf{b}$ with $\mathbf{b}\in \nat^n$. This means that $b+j_1+\cdots +j_i=\mathbf{c}$, which implies that $j_1+\cdots +j_i \leq \mathbf{c}$, in each component, i.e. $\{j_1,\dots ,j_i\}\in \Gamma_{\mathbf{c}}$, moreover $\{j_1,\dots ,j_i\}$
is also a face of the $(j-1)$-skeleton, because we were supposing that $i\leq j$. 
Therefore, it is enough to consider the isomorphism that sends $m e_{j_1}\wedge \cdots \wedge e_{j_i}$ to the face $\{j_1,\dots ,j_i\}$. Then, it is easy to see that the differentials
in the two complexes are defined in the same way.
\end{proof}

\begin{Rem}
When $k<d$, we have that $ \tilde{\operatorname{H}}_{i}(\Gamma_{\mathbf{c}}^{\langle j-1 \rangle}, \Ffield)\cong \tilde{\operatorname{H}}_{i}(\Gamma_{\mathbf{c}}, \Ffield)$, for all $\mathbf{c}$ with $|\mathbf{c}|=k+jd$. 
\end{Rem}

\begin{ex}
  We are going to show that $\beta_{2,10}(S_{2,3,4})\neq 0$. 
  Here, $k=4>d=3$ and we show that $\beta_{2,(3,7)}(S_{2,3,4})=1$ by computing the first reduced homology of  $\Gamma_{(3,7)}^{\langle 1 \rangle}$.
  In Figure \ref{fig:gamma58withoutskeletonreduction} we show $\Gamma_{(3,7)}$ and its $1$-skeleton.  
  \begin{figure}[htb]
  \centering
  \begin{tabular}{cc}
  \begin{tikzpicture}[thick, scale=0.5]


  \coordinate (0) at (0,0);
  \coordinate (1) at (0,2);
  \coordinate (2) at (3,2);
  \coordinate (3) at (3,-2);
  
  \node[left] at (0) {\scriptsize{$(0,3)$}};
  \node[left] at (1) {\scriptsize{$(3,0)$}};
  \node[right] at (2) {\scriptsize{$(2,1)$}};
  \node[right] at (3) {\scriptsize{$(1,2)$}};

  \draw (1) -- (0);
  \draw[fill=cof,opacity=0.6] (0) -- (2) -- (3) -- cycle;
  \end{tikzpicture}&\begin{tikzpicture}[thick, scale=0.5]
 \coordinate (0) at (0,0);
  \coordinate (1) at (0,2);
  \coordinate (2) at (3,2);
  \coordinate (3) at (3,-2);
  
  \node[left] at (0) {\scriptsize{$(0,3)$}};
  \node[left] at (1) {\scriptsize{$(3,0)$}};
  \node[right] at (2) {\scriptsize{$(2,1)$}};
  \node[right] at (3) {\scriptsize{$(1,2)$}};

  \draw (1) -- (0);
  \draw (0) -- (2) -- (3) -- cycle;

    \end{tikzpicture}
  \end{tabular}
    \caption{\scriptsize{On the left side: the complexes $\Gamma_{(3,7)}$. On the right side: its $1$-skeleton $\Gamma_{(3,7)}^{\langle 1 \rangle}$.}}\label{fig:gamma58withoutskeletonreduction}
  \end{figure}
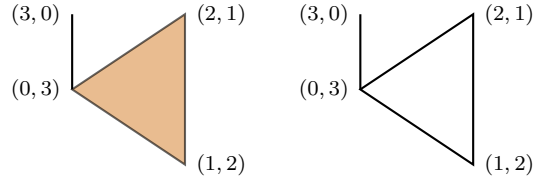

\end{ex}

As we remarked in the previous proof the Betti number $\beta_{i, \mathbf{c}}=0$ if $|\mathbf{c}|\neq k+jd$ for all $j$. 
So, in the rest of the paper, we will consider the following  more \emph{compact} version of the Betti table.
\begin{scriptsize}
\[
  \begin{array}{c|cccccc}
    &0 &\dots  &(N-n) &(N-n+1)  &\dots&(N-1)\\\hline
    k &\beta_{0,k} &\dots&\beta_{N-n, k+(N-n)d}&\beta_{N-n+1, k+(N-n+1)d}&\dots& \beta_{N-1, k+(N-1)d}\\
    
    k+1 &\beta_{0,k+d} & \dots&\beta_{N-n, k+(N-n+1)d}& \beta_{N-n+1, k+(N-n+1+1)d}&\dots&\beta_{N-1, k+Nd}\\
    
    \vdots &&&&&&\\

    k+i &\beta_{0,k+id} & \dots&\beta_{N-n, k+(N-n+i)d}&\beta_{N-n+1, k+(N-n+1+i)d}&\dots&\beta_{N-1, k+(N-1+i)d}\\

    \vdots &&&&&&\\

    k+n-1&\beta_{0,k+(n-1)d} & \dots&\beta_{N-n, k+(N-1)d}&\beta_{N-n+1, k+Nd}&\dots&\beta_{N-1, k+(N+n-2)d}
  \end{array}
\]
\end{scriptsize}


\begin{Rem}\label{theorem-regularity}
  As a direct consequence of Theorem \ref{Theorem-theoretical-results} and of equation (\ref{duality}), one can show that the \emph{compact} Betti diagram of $S_{n,d,k}$ has at most $n$ rows. (This gives a bound on the regularity.)
  This can be also obtained by means of Koszul cohomology.
\end{Rem}
  %

It is  known that $S_{n,d,0}$ is Cohen-Macaulay (see \cite[Proposition 9]{Spaul}) and it is possible to study the projective dimension of $S_{n,d,k}$ via local cohomology techniques (see Chapter 3 in \cite{BrunsHerzog}).
Let us  characterize the Cohen-Macaulayness of the Veronese modules $S_{n,d,k}$, using our combinatorial approach.

\begin{teo}\label{cm}
The Veronese module $S_{n,d,k}$ is Cohen-Macaulay if and only if $k<d$.
Moreover if $S_{n,d,k}$ is not Cohen-Macaulay, then it has maximal projective dimension, $\operatorname{pdim} S_{n,d,k} = N-1$.
\end{teo}
\begin{proof}
We know that $N-n\leq \operatorname{pdim} S_{n,d,k}\leq N-1$.
We are going to show that if $k<d$ then $\operatorname{pdim} S_{n,d,k}=N-n$ and 
if $k\geq d$ then $\operatorname{pdim}S_{n,d,k}= N-1$.

Let $k<d$.
We want to prove that $\beta_{N-n+1, k+(N-n+1+i)d}=0$ for all $i$.
For simplicity let $\alpha=(N-n+1+i)$ and $|\mathbf{c}|=k+d\alpha$.
We note that $\operatorname{dim}\Gamma_{\mathbf{c}} <\alpha$ and thus  $\Gamma_\mathbf{c}^{\langle \alpha-1 \rangle}=\Gamma_\mathbf{c}$. 
%
%
Using Theorem \ref{Theorem-theoretical-results}, we know that
\begin{equation}\label{hom}
   \beta_{N-n+1,\mathbf{c}}(S_{n,d,k})=
  \operatorname{dim}_{\Ffield} \tilde{\operatorname{H}}_{N-n}
  (\Gamma_{\mathbf{c}}^{\langle N-n+i \rangle}, \Ffield)=
  \operatorname{dim}_{\Ffield} \tilde{\operatorname{H}}_{N-n}(\Gamma_{\mathbf{c}}, \Ffield).
\end{equation}
Applying (\ref{duality}), we get
\[
   \beta_{N-n+1,\mathbf{c}}(S_{n,d,k})=\operatorname{dim}_{\Ffield} 
  \tilde{\operatorname{H}}_{N-n}(\Gamma_{\mathbf{c}}, \Ffield)=\operatorname{dim}_{\Ffield} 
  \tilde{\operatorname{H}}_{-2}(\Gamma_{\hat{\mathbf{c}}}, \Ffield)=0.
\]

Now we prove that if $k\geq d$ then $\operatorname{pdim} S_{n,d,k}= N-1$. 
To prove our statement, it is enough to show that $\beta_{N-1,k+(N-1)d}(S_{n,d,k})\neq 0$.
Since 
\[
    \beta_{N-1,k+(N-1)d}(S_{n,d,k})=\sum_{|\mathbf{c}|=k+(N-1)d}
    \operatorname{dim}_{\Ffield} 
    \tilde{\operatorname{H}}_{N-2}(\Gamma_{\mathbf{c}}^{\langle N-2 \rangle}, \Ffield),
\]
it is sufficient  to prove  that for some $\bar{\mathbf{c}}$, $\operatorname{dim}_{\Ffield} 
\tilde{\operatorname{H}}_{N-2}(\Gamma_{\bar{\mathbf{c}}}^{\langle N-2 \rangle}, \Ffield)\neq 0$.

Let $|\mathbf{c}|=k+(N-1)d$ and assume for simplicity that $|\mathbf{c}|$ is a 
multiple of $n$. 
Consider the case $\bar{\mathbf{c}}=(\frac{k+(N-1)d}{n}, \dots,\frac{k+(N-1)d}{n})$.
Then
$\Gamma_{\bar{\mathbf{c}}}^{\langle N-2 \rangle}$ is the boundary of an $N-1$ dimensional simplex. 
Indeed, we denote by $F_i$ the face of cardinality $N-1$ over the $N$ vertices with $a_i=(i_1, \dots, i_n)$ missing:
$\operatorname{deg} F_i=(\frac{(N-1)d}{n}-i_1, \dots,\frac{(N-1)d}{n}-i_n)$.
Since $\frac{k+(N-1)d}{n}\geq \frac{(N-1)d}{n} \geq \frac{(N-1)d}{n}-i_j$, $F_i\in 
\Gamma_{\bar{\mathbf{c}}}^{\langle N-2 \rangle}$ for any $i$.
%

If $|\mathbf{c}|$ is not a multiple of $n$, write $|\mathbf{c}|=mn+s$ and 
prove (in a similar way) that $\Gamma_{\mathbf{c}^*}^{\langle N-2 \rangle}$ is the boundary of an $N-1$ dimensional simplex with $\mathbf{c}^*=m\mathbf{1}+(s_1,\dots, s_n)$ and $\sum_js_j=s$, $0\leq s_j\leq 1$.
\end{proof}

\begin{Rem}
The equalities in (\ref{hom})
 hold for all $i\geq 1$ also  for $k\geq d$. This follows from the fact that for $i>0$ it is always true that
$\tilde{\operatorname{H}}_{N-n} (\Gamma_{\mathbf{c}}^{\langle N-n+i \rangle}, \Ffield)=  \tilde{\operatorname{H}}_{N-n}(\Gamma_{\mathbf{c}}, \Ffield)$.
Instead, for $i=0$, $\Gamma_\mathbf{c}$ could not be the same of $\Gamma_\mathbf{c}^{\langle N-n \rangle}$.
This implies that, for $k\geq d$, $\beta_{i,k+id}(S_{n,d,k})\neq 0$ for $i=0, \dots, N-n+1$.
\end{Rem}
\subsection{Pseudo-Linearity of $S_{n,d,k}$}

By studying the dimension of the pile simplicial complex one also obtains a sufficient condition for the pseudo-linearity of the resolution.  Namely, as  a straightforward application of Theorem \ref{Theorem-theoretical-results} and of Paul's duality formula, we can prove the following statement. Also in this case, though, Koszul cohomology can also be used to show such result.

\begin{pro}\label{theorem-linear-resolution}
If $k> d(n-1)-n$, then the Veronese module $S_{n,d,k}$ has a \emph{pseudo-linear} resolution.
\end{pro}

%
%


\begin{ex}
For $S_{3,3,3}$ we have that $k = d(n-1)-n$ and using \texttt{Macaulay2} \cite{M2} for calculating the Betti numbers one can see that the resolution is not even pure.
\end{ex}

%

\begin{Rem}\label{cor-linearity-n-2}
If $k\geq d-1$, then the Veronese module $S_{2,d,k}$ has \emph{pseudo-linear} resolution.
Moreover, by the knowledge of the Hilbert series (see Section \ref{sec-HS}) and using the fact that the resolution is pure, we have that 
$$
\beta_{i, k+id}(S_{2,d,k})=|i-(k+1)|\binom{d}{i}.
$$
  %
%
Similiarly, when $k\geq 2d-2$, then the Veronese module $S_{3,d,k}$ has \emph{pseudo-linear} resolution.
In this case the Betti numbers are
$$
\beta_{i, k+id}(S_{3,d,k})=\left|\frac{1}{2} \left(a\binom{N-3}{i}-(b-2a)\binom{N-3}{i-1}+(a-c)\binom{N-3}{i-2}\right)\right|,
$$
where $a,b,c$ have been defined in Section \ref{sec-HS}.
\end{Rem}

Using the Eagon-Northcott resolution, it is well known that $\field[x,y]^{(d)}$ has a  linear resolution.
We provide another  proof for this fact:
\begin{cor}\label{cor-veronese-2-variables}
The resolution of the Veronese subring $\field[x,y]^{(d)}$ is \emph{linear}.
\end{cor}
\begin{proof}
Since $\field[x,y]^{(d)}$ is Cohen-Macaulay, it is enough to show that $\beta_{d-1,(d+1)d}$.
By using (\ref{duality}), $\beta_{d-1,(d+1)d}=\operatorname{dim}_{\Ffield} \tilde{\operatorname{H}}_{d-2}(\Gamma_{\mathbf{c}}, \Ffield)=\operatorname{dim}_{\Ffield} \tilde{\operatorname{H}}_{-1}(\Gamma_{\hat{\mathbf{c}}}, \Ffield)=0$, since $|\hat{\mathbf{c}}|=-2$.
\end{proof}

%
In general, we are able to say that the resolution of $S_{n,d,k}$ is always pseudo-linear in the first step. In particular this is true for the canonical module of $S^{(d)}$. 
The Betti number 
$\beta_{0,i}(M)$ of
a graded module $M$ gives the number of generators of $M$ in degree $i$. Thus 
$\beta_0(S_{n,d,k})=
\beta_{0,k}(S_{n,d,k})={k+n-1\choose k}$. \\
In the following, we will prove that $\beta_1(S_{n,d,k})= \beta_{1,k+d}(S_{n,d,k})$.


\begin{defi}
 Given two vectors $v,w \in \mathbb{N}^n$, we say that $v$ is obtained from $w$ with 
 an \emph{elementary move} if $v=w+(e_i-e_j)$, where $e_i$ denotes the standard vector 
 in $\mathbb{N}^n$.
\end{defi}

\begin{pro}
The pile simplicial complexes $\Gamma_{\mathbf{c}}$ are connected for 
$|\mathbf{c}|=k+id$ for $i>1$, $k\neq 0$.
\end{pro}
\begin{proof}
We remark that $\Gamma_{\mathbf{c}}^{\langle 0 \rangle}$ is the set of vertices of 
the pile simplicial complex.
Given two vertices in $v,w\in \Gamma_{\mathbf{c}}^{\langle 0 \rangle}$ one gets $w$ from $v$
by a finite number of elementary moves and this implies that the \emph{graph of elementary moves}
on $\Gamma_{\mathbf{c}}^{\langle 0 \rangle}$ is connected.

\noindent
Therefore, it is enough to prove if $v,w \in \Gamma_{\mathbf{c}}^{\langle 0 \rangle}$ such that 
$v=w+(e_j-e_l)$, then there exists $u \in \Gamma_{\mathbf{c}}^{\langle 0 \rangle}$ such that 
$\{ v,u\}$ and $\{w,u\}$ are edges of $\Gamma_{\mathbf{c}}$.

For simplicity, set $i=2$, $j=1$ and $l=2$, i.e. $v=w+(1,-1,0,\dots ,0)$.
Such $u$ should respect the following inequalities:
\begin{eqnarray*}
    u_i+w_i &\leq& c_i, \  1\leq i\leq n.\\
    u_i+v_i &\leq& c_i, \  1\leq i\leq n.
\end{eqnarray*}
but $v_1=w_1+1$ and so the sufficient constrains are
\begin{eqnarray*}
    u_1+w_1 &<&c_1,\\
    u_i+w_i &\leq& c_i, \  2\leq i\leq n.
\end{eqnarray*}
Let us choose $u_i=c_i-w_i$, $2\leq i \leq n$, and $u_1=d-\sum_{2}^n u_i$: note that $u_1=c_1-k-id+d-w_1$, so $u_1+w_1=c_1+d(2-i)-k<c_1$, when $k\neq 0 $ and $i\geq 2$. This choice does not work in case $u_1<0$, that is $d+w_2+\cdots +w_n<c_2+\cdots +c_n$. In this circumstance, we set $u_1=0$, $u_i=\min \{d-\sum_{j=1}^{i-1} u_j, c_i-w_i\}$ and $u_n=d-\sum_{j=1}^{n-1} u_j$, and this vector $u$ satisfies the inequalities above. Namely, it can happen that $u_j=c_j-w_j$, for $2\leq j < i$, and $u_i=d-u_1-\cdots -u_{i-1}$, then in this case $u_{i+1}=\cdots = u_n=0$, and the inequalities are satisfied.  The other case is when $u_j=c_j-w_j$ for $2\leq j <n-1$ and $u_n=d-u_1-\cdots -u_n$, then $u_n+w_n=
d+w_2+\cdots +w_n-c_2-\cdots - c_{n-1}<c_n$, and the other inequalities are also trivially satisfied.
\end{proof}

\begin{cor}\label{cor-beta-1-general}
  If $k> 0$, then $\beta_{1,k+d}(S_{n,d,k})\neq 0$ and that $\beta_{1,k+id}(S_{n,d,k})=0$, for all $i> 1$.
\end{cor}
\begin{proof}
  The statement follows using Theorem \ref{Theorem-theoretical-results} and from the previous proposition applied to $\beta_{1,k+id}(S_{n,d,k})$.
  %
\end{proof}

\begin{cor}\label{cor-linearity-general-n-d-k}
  For $k\neq 0$, the first step of the resolution of $S_{n,d,k}$ is pseudo-linear, and the first three entries of the Betti table, $\beta_{0,k}, \beta_{1,k+d}$ and $\beta_{2,k+2d}$, can be then determined with the Hilbert series.
\end{cor}
\noindent
The previous result can be derived using Koszul cohomology. Nevertheless, this combinatorial proof shows how the  topological properties of the pile simplicial complex can imply trivially algebraic features of the Veronese modules.

%

\section{The resolution of $S_{2,d,k}$}\label{sec-n=2}
%
%
Let us consider the Veronese modules $S_{2,d,k}$.
By Theorem \ref{cm} these modules are Cohen-Macaulay if and only if $k<d$. 
In this section, we are going to give a full description of their Betti diagram using the methods introduced in the previous sections.

It is important to remark that this result can be also obtained as a consequence of Corollary 3.a.6 and Theorem 1.b.4. in  \cite{Green_koszul}, by means of Koszul cohomology.
\begin{teo}\label{theo-jump}
  If $k<d$, the Veronese module $S_{2,d,k}$ has pure resolution and the Betti table is:
  \[
  \begin{array}{c|cccccccc}
    &0 &1 &\dots &k &k+1 &k+2    &\dots& d-1\\\hline
    k &k+1 &k\binom{d}{1}& \dots &\binom{d}{k}&0&0&\dots&0\\
 
    k+1 &0 &0& \dots&0&\binom{d}{k+2}&2\binom{d}{k+3}&\dots&(d-1-k)\binom{d}{d}
  \end{array}
  \]
  Namely,  $\beta_{i}(S_{2,d,k})=\beta_{i, k+id}(S_{2,d,k})=(k+1-i)\binom{d}{i}$ for $i\leq k$, and $\beta_{i}(S_{2,d,k})=\beta_{i, k+(i+1)d}(S_{2,d,k})=(i-k)\binom{d}{i+1}$ for $i>k$.
\end{teo}
\noindent

For a given $f$-dimensional simplicial complex $\Delta$, we denote by $\Delta(i)$ the $i$-th pure skeleton of $\Delta$, for $0\leq i \leq f$.

\begin{lem}\label{lemma-v0-in-pure-part}
  Let $|\mathbf{c}|=|(c_1,c_2)|=k+(k+1)d$ and $c_1\leq c_2$. 
  If $\operatorname{dim} \Gamma_{\mathbf{c}}=k$, then $\mathbf{a}_0=(0,d)$ is a vertex of $\Gamma_{\mathbf{c}}(k)$.
\end{lem}
\begin{proof}
  We observe that if $k\geq d-1$, then the statement is trivially true.
  So let us assume that $k\leq d-1$.
  
  We denote by $\mathbf{a}_{i}=(i, d-i)$.
  Since $\operatorname{dim} \Gamma_{\mathbf{c}}=k$, then there exists a $k$-dimensional face (facet) $$F=\{\mathbf{a}_{j_0}, \mathbf{a}_{j_1}, \dots, \mathbf{a}_{j_{k}}\},$$
  where $j_0<j_1<\dots <j_k$.
  If $j_0=0$, then $\mathbf{a}_0\in F$ and this implies $\mathbf{a}_0\in \Gamma_{\mathbf{c}}(k)$.
  
  Let us assume that $\mathbf{a}_0\notin F$.
  Since $F\in \Gamma_{\mathbf{c}}(k)$, then
  \begin{eqnarray}
    j_0+j_1+\dots+j_{k}&=&c_1-m_1,\\\label{eq:suc1}
    (k+1)d-(j_0+j_1+\dots+j_{k})&=&c_2-m_2,\label{eq:suc1}
  \end{eqnarray}
  for some $m_i\in \nat$ such that $m_1+m_2=k$.
  
  We are going to show that there exists a facet $G\in\Gamma_{\mathbf{c}}(k)$ with $m_1=0$ (and so $m_2=k$).
  To do that we increase as much as possible $j_k$, $j_{k-1}$, etc (notice that $j_k\leq d$, $j_{k-1}\leq d-1$ and $j_{k-l}\leq d-l$).
  This is possible in all the cases where
  \begin{equation}\label{eq:spaziosopra}
    \sum_{l=0}^{k} (d-l)-j_{k-l}>k,
  \end{equation}
  because $m_1$ can be at most $k$. 
  
  The condition $c_1\leq c_2$, implies equation (\ref{eq:spaziosopra}).
  Indeed, if $\sum_{l=0}^{k} (d-l)-j_{k-l}=k$, then 
  \begin{eqnarray*}
    c_1&=&(k+1)d -\frac{k(k+1)}{2}-k+m_1,\\\label{eq:suc1}
    c_2&=&k+\frac{k(k+1)}{2}+m_2.\label{eq:suc1}
  \end{eqnarray*}
  Imposing $c_1\leq c_2$, one gets $d\leq k+1$, against the initial assumption.
  
  Let $G=\{\mathbf{a}_{g_0}, \mathbf{a}_{g_1}, \dots, \mathbf{a}_{g_{k}}\}$.
  If $g_0\leq k$, then we consider $$H=(G\setminus\{\mathbf{a}_{g_0}\}) \cup \{\mathbf{a}_{0}\}$$
  and this facet belongs to $\Gamma_{\mathbf{c}}(k)$.
  
  If $g_0 > k$, we produce a facet $L=\{\mathbf{a}_{l_i}: 0\leq i\leq k\}$ with $l_0\leq k$, $m_1=0$ and $m_2=k$.
  To do that we decrease the value of $g_0$ by one and we add one to $g_k$, if $g_k$ is not maximal, or to $g_{k-1}$, if $g_{k-1}$ is not maximal, etc.
  We repeat this procedure until $g_0\leq k$.
  
  To show that this method always work, note that $g_0-k$ is at most $d-k$.
  We need that
  \begin{equation}\label{eq:spaziosopra2}
    \sum_{l=0}^{k-1} (d-l)-j_{k-l}>d-k.
  \end{equation}
  Again, with similar computation, the condition $c_1\leq c_2$, implies equation (\ref{eq:spaziosopra2}).
\end{proof}

\begin{proof}[Proof of Theorem \ref{theo-jump}]
From the Hilbert series, we know that $\beta_{k, k+(k+1)d}-\beta_{k+1,k+(k+1)d}=0$. We are going to prove that $\beta_{k+1,k+(k+1)d}=0$, where $\beta_{k+1,k+(k+1)d}=\sum_{|\mathbf{c}|=k+(k+1)d} \operatorname{dim}_{\Ffield} \tilde{\operatorname{H}}_{k}(\Gamma_{\mathbf{c}}, \Ffield)$.

If for some $\mathbf{c}$, $\operatorname{dim}(\Gamma_\mathbf{c})<k$, then the $k$-th reduced homology is $0$. In the other cases, we can use the fact that $\tilde{\operatorname{H}}_{k}(\Gamma_{\mathbf{c}}, \Ffield)\cong \tilde{\operatorname{H}}_{k}(\Gamma_{\mathbf{c}}(k), \Ffield)$. Let us denote by $\Delta_{d,k}$ the $k$-th pure skeleton $\Gamma_{\mathbf{c}}(k)$.

 By Lemma \ref{lemma-v0-in-pure-part}, we know that $\mathbf{a}_0\in \Delta_{d,k}$.
 
 \noindent
 The idea of the proof is to decompose $\Delta_{d,k}$ into the union of two subcomplexes and then relate their reduced homology with the one of $\Delta_{d,k}$. Namely,  
 $$\Delta_{d,k} = C \cup D,$$
 with $C=\{ F\in \Delta_{d,k} | \ \operatorname{dim}(F)=k \ \mathrm{and} \ \mathbf{a}_0\in F \}$ and $D=\{ F\in \Delta_{d,k} | \ \operatorname{dim}(F)=k \ \mathrm{and} \ \mathbf{a}_0\notin F \}$.\\
 We will prove by induction on $k$ and $d$ that  the $k$-th reduced homology of $\Delta_{d,k}$ is zero. To do this, by Mayer-Vietoris long exact sequence, it is enough to prove that both $D$ and $C$ have zero $k$-th reduced homology and that their intersection has zero $(k-1)$-reduced homology.
The bases of this double induction are  $k=0$, and $d=1$: when  $k=0$, $S_{2,d,0}$ is the Veronese ring, and the assertion is trivial; for $d=1$, since $k<d$ this implies the case $k=0$. \\
Now let us consider $k$ and $d$, with $k<d$, and let us assume that $\tilde{\operatorname{H}}_{k}(\Delta_{d,k}, \Ffield)\cong 0$  for all $k'<k$ and $d'<d$. \\
The complex $C$ is always a cone over $\mathbf{a}_0$, thus $\tilde{\operatorname{H}}_{j}(C,\Ffield)\cong 0$, for all $j$.\\
The simplicial complex $D$ has zero $k$-th reduced homology by induction on $d$. A facet $\{ \mathbf{a}_{i_0}, \dots , \mathbf{a}_{i_k} \}$ contained in $D$ satisfies the inequalities:
\[
c_1-k \leq  i_0+\cdots + i_k \leq c_1.
\]
Moreover, since $\mathbf{a}_0\notin D$, $i_j\geq 1$, for all $j$. Thus, we may consider $\mathbf{b}_{i_0}=\mathbf{a}_{i_0}-(1,0)=(i_0-1,d-i_0), \dots , \mathbf{b}_{i_k}=\mathbf{a}_{i_k}-(1,0)=(i_k-1,d-i_k)$, and we get that
$\{ \mathbf{a}_{i_0}, \dots , \mathbf{a}_{i_k} \}\in D$ if and only if $\{ \mathbf{b}_{i_0}, \dots , \mathbf{b}_{i_k} \}$ belongs to the simplicial complex $\Gamma_{\mathbf{c}-(k+1,0)}$ (on vertex set $\{ (0,d-1), (1,d-2), \dots , (d-1,0)\}$)(i.e. $|\mathbf{c}-(k+1,0)|=k+(k+1)(d-1)$). Thus $\tilde{\operatorname{H}}_{k}(D,\Ffield)\cong 0$  by induction on $d$.\\

Now, let us prove that the $(k-1)$-th reduced homology of $C\cap D$ is zero. First of all, notice that $\operatorname{dim}(C\cap D)\leq k-1$. If $\operatorname{dim}(S\cap D)<k-1$, then we are done. Thus we assume that $\operatorname{dim}(C\cap D)=k-1$. Let us consider a facet $F=\{ \mathbf{a}_{i_1}, \dots , \mathbf{a}_{i_k}  \}\in C\cap D$. Since $F\in D$, there exists $j\neq 0$ such that $F\cup \{\mathbf{a}_j\}\in D$, thus $F$ satisfies the inequality:
$$
i_1+\dots +i_k\leq c_1-1.
$$
On the other hand, $F\in C$, so $F\cup\{ \mathbf{a}_0\}$, thus $c_1-k\leq i_1+\dots +i_k$.\\
Thus if we consider $(c_1-1,c_2-d)$ we have that $S\cap D$ is isomorphic to $\Gamma_{(c_1-1,c_2-d)}(k-1)$, with $|(c_1-1,c_2-d)|=k-1+kd$. Thus, by induction on $k$, $\tilde{\operatorname{H}}_{k-1}(C\cap D, \Ffield)\cong 0$.
\end{proof}

\section{Applications in three and four variables}
In this section we calculate the Betti tables of the Veronese rings, $S_{3,4,0}$, $S_{3,5,0}$ and $S_{4,3,0}$,  using the knowledge of the Hilbert series and the regularity (see Remark \ref{theorem-regularity}).
By duality, these give also the Betti tables of their canonical modules, $S_{3,4,1}$, $S_{3,5,2}$ and $S_{4,3,2}$.

\noindent
The number of rows of the Betti tables for these Veronese rings and modules is  three.
By using \texttt{Macaulay2}, we are able to compute the first \emph{row} of the Betti table of  $S_{3,4,1}$, $S_{3,5,2}$ and $S_{4,3,2}$.
Finally, by knowing the polynomials $P(z)$ of $S_{3,4,0}$, $S_{3,5,0}$ and $S_{4,3,0}$, we derive all their Betti numbers. 
The results obtained support the conjecture that the resolution of the Veronese ring is linear until homological degree $3d-3$ (see \cite{EinLazarsfeld,OttavianiPaolettiVeronese}). 
 
\paragraph{The Betti table of $S_{3,4,0}$.}
The first row  of the Betti table of  $S_{3,4,1}$ is:
  \begin{scriptsize}
  \begin{verbatim}
            0  1  2 
     total: 3 24 55 
         1: 3  .  . 
         2: .  .  . 
         3: .  .  . 
         4: . 24  . 
         5: .  .  . 
         6: .  .  . 
         7: .  . 55          
  \end{verbatim}
  \end{scriptsize}
  
\noindent
  We compute the polynomial  $P(z)$ of $S_{3,4,0}$:
  \[
     \begin{split}
     h(z)=1-75\,{z}^{2}+536\,{z}^{3}-1947\,{z}^{4}+4488\,{z}^{5}-7095\,{z}^{6}+7920\,{z}^{7}+\\
    -6237\,{z}^{8}+3344\,{z}^{9}-1089\,{z}^{10}+120\,{z}^{11}+55\,{z}^{12}-24\,{z}^{13}+3\,{z}^{14}.
    \end{split}
  \]
  Hence the compact Betti table  of $S_{3,4,0}$ is:
  \begin{scriptsize}
  \begin{verbatim}
            0  1    2     3     4     5     6     7     8     9   10  11 12
         0: 1  .    .     .     .     .     .     .     .     .    .   .  .
         1: . 75  536  1947  4488  7095  7920  6237  3344  1089  120   .  .
         2: .  .    .     .     .     .     .     .     .     .   55  24  3  
  \end{verbatim}
  \end{scriptsize}
 \paragraph{The Betti table of $S_{3,5,0}$.}
  The first row of the Betti table of  $S_{3,5,2}$, in the compact form, is:
  \begin{scriptsize}
  \begin{verbatim}
            0   1    2     3     4     5  6
         2: 6  90  595  2160  4200  2002  .   
  \end{verbatim}
  \end{scriptsize}
  
\noindent
  The polynomial $P(z)$ of $S_{3,5,0}$ using  the formula in Section \ref{sec-HS} is:
  \begin{scriptsize}
  \[
    \begin{split}
    &h(z)= 1-165{z}^{2}+1830{z}^{3}-10710{z}^{4}+41616{z}^{5}-117300{z}^{6}+250920{z}^{7}+\\
    &	  -417690\,{z}^{8}+548080\,{z}^{9}-568854\,{z}^{10}+464100\,{z}^{11}-291720\,{z}^{12}+134640\,{z}^{13}+\\
    &	  -39780\,{z}^{14}+2856\,{z}^{15}+3825\,{z}^{16}-2160\,{z}^{17}+595\,{z}^{18}-90\,{z}^{19}+6\,{z}^{20}.
    \end{split}
  \]
  \end{scriptsize}
  \newline
  \noindent
  The following is the compact Betti table of $S_{3,5,0}$.
  \begin{scriptsize}
  \begin{verbatim}
            0   1     2  ...      12     13    14    15    16   17 18
         0: 1   .     .    .       .      .     .     .     .    .  .
         1: . 165  1830  ...  134640  39780  4858   375     .    .  .
         2: .  .      .    .       .   2002  4200  2160   595   90  6  
  \end{verbatim}
  \end{scriptsize}

\paragraph{The Betti table of $S_{4,3,0}$.} First of all, since $S_{4,3,0}$ is Cohen-Macaulay, we prove that $\beta_{16, 19\cdot 3}=0$. 
  Indeed, using (\ref{duality}), 
  $$\beta_{16, 57}=\sum_{|\mathbf{c}|=57}\operatorname{dim}\tilde{H}_{15}(\Gamma_{\mathbf{c}})=\sum_{|\mathbf{c}|=-1}\operatorname{dim}\tilde{H}_{-1}(\Gamma_{\mathbf{c}})=0.$$
  These implies that the compact form of the Betti table has only three rows.\\  
  In this case, we compute the polynomial $P(z)$ of $S_{4,3,0}$ by differentiating $\nicefrac{z^3}{(1-z^3)}$ (see Corollary \ref{cor-close-formula-HS}) and by multiplying the numerator by $(1-z^3)^{16}$:
  \begin{scriptsize}
  \[
    \begin{split}
	 & h(z)=1-126\,{z}^{2}+1200\,{z}^{3}-5940\,{z}^{4}+19152\,{z}^{5}-43680\,{z}^{6}+73008\,{z}^{7}-90090\,{z}^{8}+80080\,{z}^{9}+\\
	 & -46332\,{z}^{10}+9360\,{z}^{11}+12012\,{z}^{12}-15120\,{z}^{13}+9360\,{z}^{14}-3696\,{z}^{15}+945\,{z}^{16}-144\,{z}^{17}+10\,{z}^{18}.    
    \end{split}
\]
  \end{scriptsize}
\noindent  
 The first \emph{line} of the compact Betti diagram of $S_{4,3,2}$ is:
  \begin{scriptsize}
  \begin{verbatim}
             0    1    2     3     4      5      6     7     8    9  10
         1: 10  144  945  3696  9360  15120  14003  5400  1650  220   .   
  \end{verbatim}
  \end{scriptsize}  
\noindent
  Finally the Betti numbers of $S_{4,3,0}$ are:
  \begin{scriptsize}
  \begin{verbatim}
            0   1  ...      6      7      8      9      10     11    12    13   14   15  16
         0: 1   .  ...      .      .      .      .       .      .     .     .    .    .   .
         1: . 126  ...  73008  90090  80300  47982   14760   1991     .     .    .    .   .
         2: .   .  ...      .    220   1650   5400   14003  15120  9360  3696  945  144  10  
  \end{verbatim}
  \end{scriptsize}
In principle, the procedure could be extended to other cases, $S_{3,d,0}$ with $d\geq 6$ (since the regularity is still $2$), or even to cases, like $S_{4,d,0}$ with $d\geq 4$, where the regularity is $3$: these last cases could be obtained by calculating the first row of the Betti diagram of the Veronese ring, the first row of the Betti diagram of its canonical module, and, by the knowledge of the Hilbert series, one could get the second row of the Betti table of the Veronese ring. Unfortunately, with our computers (see Section \ref{sec-HS} for more details about our equipment), we were not able to go any further in computations.\\


\noindent
\textbf{Acknowledgements}\\
We would like to thank Ralf Fr\"oberg and Mats Boij, for their great support and useful suggestions. \\
We thank the anonymous referee for the useful comments.\\
Some computations in this paper were performed by using Maple(TM); Maple is a trademark of Waterloo Maple Inc. Also, some computations were made using Macaulay 2 (see \cite{M2}).

\begin{small}

\addcontentsline{toc}{section}{Bibliography}
\bibliographystyle{siam}
\bibliography{Veronese}

 \end{small}

\noindent
 {\scshape Ornella Greco}\\ 
 {\scshape Department of Mathematics, Royal Institute of Technology, S-10044 Stockholm, Sweden}.\\
 {\itshape E-mail address}: \texttt{ogreco@kth.se}\\
 
\noindent
 {\scshape Ivan Martino}\\
 {\scshape D\'{e}partment de Math\'{e}matiques, Universit\'{e} de Fribourg, CH-1700 Fribourg, Suisse}.\\
 {\itshape E-mail address}: \texttt{ivan.martino@unifr.ch}

\end{document}